\documentclass[12pt,leqno,twoside]{article}

\usepackage{amsmath,amsfonts,amssymb,fancyhdr,amsthm,titlesec}

\newtheoremstyle{thm}{.5\baselineskip}{.5\baselineskip}{\itshape}{5mm}
{\scshape}{.}{2mm}{}

\theoremstyle{thm}
\newtheorem{theorem}{Theorem}
\newtheorem{corollary}[theorem]{Corollary}
\newtheorem{lemma}[theorem]{Lemma}
\newtheorem{proposition}[theorem]{Proposition}

\newtheoremstyle{rem}{.5\baselineskip}{.5\baselineskip}{}{5mm}
{\scshape}{.}{2mm}{}

\theoremstyle{rem}
\newtheorem{remark}[theorem]{Remark}

\renewcommand{\abstractname}{{\bf Abstract.\ }}
\renewenvironment{abstract}{\footnotesize\abstractname}

\makeatletter

\renewcommand{\section}{\@startsection{section}{1}{0pt}{\baselineskip} 
{.5\baselineskip}{\centering\normalfont\bfseries}}

\renewcommand{\subsection}{\@startsection{subsection}{1}{7mm}{\baselineskip}
{-3mm}{\bfseries}}

\renewcommand*{\@seccntformat}[1]{%
  \csname the#1\endcsname.\ }

\makeatother

\begin{document}

\pagestyle{fancy}

\thispagestyle{empty}

\begin{center}
{\bf\large EXACT KRONECKER CONSTANTS OF \\[2mm]
THREE ELEMENT SETS}

\bigskip

{\small KATHRYN E.\ HARE$^1$ and L.\ THOMAS RAMSEY$^2$\footnote{\parindent=6mm{\it Key words and phrases:} Kronecker constant, trigonometric approximation. 

{\it Mathematics Subject Classification:} Primary: 42A10; Secondary: 43A46, 11J71

This work was partially supported by NSERC through application number 44597,  and the Edinburgh Math
Soc. The authors would like to thank St. Andrews University for their
hospitality when some of this research was done.}}

\medskip

{\scriptsize $^1$Department of Pure Mathematics, University of Waterloo, Waterloo, Ont., Canada, N2L 3G1 \\[-2mm]
e-mail: kehare@uwaterloo.ca}

\medskip

{\scriptsize $^2$Department of Mathematics, University of Hawaii at Manoa, Honolulu, Hi, USA, 96822 \\[-2mm]
e-mail: ramsey@math.hawaii.edu}
\end{center}

\bigskip

\begin{quotation}
\begin{abstract}
For any three element set of positive integers, $\{a,b,n\}$, with $a<b<n,$ $%
n $ sufficiently large and $\gcd (a,b)=1$, we find the least $\alpha $ such
that given any real numbers $t_{1},t_{2},t_{3}$ there is a real number $x$
such that 
\begin{equation*}
\max \{\left\langle ax-t_{1}\right\rangle ,\left\langle
bx-t_{2}\right\rangle ,\left\langle nx-t_{3}\right\rangle \}\leq \alpha ,
\end{equation*}%
where $\left\langle \cdot \right\rangle $ denotes the distance to the
nearest integer. The number $\alpha $ is known as the angular Kronecker
constant of $\{a,b,n\}$. We also find the least $\beta $ such that the same
inequality holds with upper bound $\beta $ when we consider only
approximating $t_{1},t_{2},t_{3}$ $\in \{0,1/2\}$, the so-called binary
Kronecker constant. \ The answers are complicated and depend on the
congruence of $n\mod(a+b)$. Surprisingly, the angular and binary
Kronecker constants agree except if $n\equiv a^{2}\mod(a+b)$.
\end{abstract}
\end{quotation}

\section{Introduction}

The classical Kronecker theorem states that if $\{r_{j}\}$ is any finite
collection of rationally independent real numbers, then given any sequence
of real numbers $(t_{j})\subseteq \lbrack 0,1)$ and $\varepsilon >0$ there
are integers $x$ and $(k_{j})$ such that $|r_{j}x-t_{j}-k_{j}|<\varepsilon $
for all \thinspace $j$. This fails to be true if the $\{r_{j}\}$ are
replaced by a finite collection of integers, $\{n_{j}\},$ even allowing $x$
to be any real number. The angular Kronecker constant of the given set of
integers $S=\{n_{j}\}$, denoted $\alpha (S)$, is the infimum of the $%
\varepsilon $ for which such an approximation can be made for every
sequence\ $(t_{j})$. It is obvious that $\alpha (S)\leq 1/2$ for any set $S$
(finite or infinite), and an application of the Baire category theorem shows
that $\alpha (S)<1/2$ for any finite set $S$. Without further knowledge
about the set, this is the best that can be said. Sets for which $\alpha
(S)<1/2$ have been much studied; c.f. \cite{GaH}-\cite{GHK}, \cite{KR}, \cite%
{Va} and the references cited therein.

The case when $\alpha (S)<1/4$ is of particular interest as such sets are
Sidon, meaning that every bounded $S$-function is the restriction of the
Fourier transform of a measure on the circle. In fact, this measure can be
chosen to be discrete and $1/4$ is sharp with this property (\cite{GH1}).
Sidon sets have been thoroughly studied, yet fundamental problems remain
open such as whether every Sidon set is a finite union of sets that have
small Kronecker constants, or whether $\alpha <1/2$ implies $S$ is Sidon.\ 

It could be helpful to know the Kronecker constants of finite sets in
addressing these questions as the Kronecker constant of an infinite set is
the supremum of the constants of its finite subsets. This is a difficult
problem and complete answers are known only for special sets, such as two
element sets \cite{GH1}, three element sets satisfying certain simple
relations \cite{HR1} and geometric sequences of the form $%
\{m^{j}\}_{j=0}^{d} $ for an integer $m$ \cite{HR2}.

In this note we determine the Kronecker constants for all three element sets
of positive integers, $\{a,b,n\},$ where $a$ and $b$ are coprime and $n$ is
sufficiently large. The answers are surprisingly complicated, with different
formulas depending on the congruence of $n\mod (a+b)$. Our proof is
algorithmic. Interestingly, we show that one can find the best approximate
for a given triple of real numbers $(t_{1},t_{2},t_{3})$ by either starting
with the best or the `second best' approximate for the pair $(t_{1},t_{2})$
relative to the two element set $\{a,b\},$ and then, if necessary, making a
slight modification. We call this the `greedy algorithm'; see Section 2 for
more details.

A related problem is to determine how well the set of integers $S$ can
approximate all $\{0,1/2\}$-valued sequences $(t_{j})$. The least $%
\varepsilon $ for this approximation problem is known as the binary
Kronecker constant, $\beta (S)$. Like the angular Kronecker constant, this
constant is also known only in a few special cases such as arithmetic
progressions \cite{HR3}. We calculate the binary Kronecker constants for
three element sets, as well. An unexpected fact is that $\alpha
(a,b,n)=\beta (a,b,n)$ provided $n\not\equiv a^{2}\mod (a+b).$

Here is our main theorem.

\begin{theorem}
\label{T:three} Suppose $a<b<n$ are positive integers with $\gcd (a,b)=1$
and $n$ suitably large. Assume 
\begin{equation*}
aT\equiv 1\mod (a+b),\quad n\equiv r\mod (a+b),\quad \text{and}\quad R\equiv
rT\mod (a+b)
\end{equation*}%
for $r$, $R$ and $T$ $\in \lbrack 0,a+b)$. Then 
\begin{equation*}
\alpha (a,b,n)=%
\begin{cases}
\dfrac{n+a^{2}+ab-aR}{2(a+b)(a+n)} & \text{if $0\leq R<a$} \\ 
\dfrac{n+ab}{2(an+bn+ab)} & \text{if $R=a$} \\ 
\dfrac{n+bR}{2(a+b)(b+n)} & \text{if $a<R\leq 2a$} \\ 
\dfrac{n+2a^{2}+2ab-aR}{2(a+b)(a+n)} & \text{if $R>2a$}%
\end{cases}%
\end{equation*}%
Moreover, $\alpha (a,b,n)=\beta (a,b,n)$ except if $R=a$, when $\alpha
(a,b,n)>\beta (a,b,n)$.
\end{theorem}

We note that $R=a$ if and only if $n\equiv a^{2}\mod (a+b)$. We also remark
that if one checks the details of the proof, it can be seen how large $n$
needs to be relative to the size of $a$ and $b$.

We conjecture that if $\gcd (a,b)=m>1$, there will be $(a+b)/m$ cases for $%
\alpha (a,b,n)$, determined by the congruence of $n$ $\mod (a/m+b/m).$

It seems reasonable to expect that with considerably more work the
techniques of this paper could be generalized to sets $S$ of size greater
than $3$, subject to the condition that the largest element is much larger
than the others.

\begin{remark}
In discrete optimization, the closest vector problem is to find, given an
additive subgroup $L\subseteq \mathbb{Z}^{n}$, the distance $d(v,L)=\min
\{\rho (v-k):k\in L\}$ for any $v\in \mathbb{R}^{n}$. This is known to be
NP-hard for both the Euclidean and maximum norms $\rho $ \cite[p. 182]{NW}.

Finding Kronecker constants involves superimposing one additional layer of
optimization. Indeed, if $S=\{n_{1},...,n_{d}\}\subseteq \mathbb{Z}^{d}$,
then $\alpha (S)=\max \{d(v,\mathbb{Z}^{d}):v\in \mathbb{R}^{d}\},$ where $%
\rho (w)=\inf \{\left\Vert w-t(n_{1},...,n_{d})\right\Vert :t\in \mathbb{R\}}
$. To date, the authors do not know the hardness level created by
superimposing the additional level of optimization that is required to
compute Kronecker constants. It is striking that with $S=\{a,b,n\}$ the
Kronecker constant can now be computed instantly for $n$ large and $\gcd
(a,b)=1$, as proved in this paper, but that exact Kronecker constants have
eluded simplification for $n$ relatively small. This is a strange kind of
`hardness', where smaller values of integers give the most difficult cases
to analyze.
\end{remark}

Throughout the paper we will denote by $E_{n}$ the value on the right hand
side of the (claimed) formula for $\alpha (a,b,n)$. In Section 2 we prove
that $\alpha (a,b,n)$ is dominated by the formulas specified above. In
Section 3 we compute $\beta (a,b,n)$, showing that it agrees with these
formulas when $R\neq a$. Since $\alpha (a,b,n)\geq \beta (a,b,n)$, this
proves the equality when $R\neq a$. The proof that $\beta (a,b,n)<\alpha
(a,b,n)=E_{n}$ when $R=a$ is handled directly in the final subsection.

\textbf{Notation and Definitions}: Assume $S=\{n_{j}\}_{j=1}^{d}$ is a set
of $d$ integers with $n_{1}<n_{2}<\cdot \cdot \cdot <n_{d}$. Set $\mathbf{n}%
=(n_{1},\hdots,n_{d})$. We define the approximation cost for $\mathbf{t}%
=(t_{1},...,t_{d})\in \mathbb{R}^{d}$, relative to $S,$ as%
\begin{equation}
\begin{aligned} \mu_S(\mathbf {t})&=\inf \{\,{\Vert \mathbf {t}-x \cdot
\mathbf n+\mathbf k\Vert_\infty}\,:\, {x\in \mathbb R, \mathbf k \in \mathbb
Z^d}\,\}\\ &=\inf\{\,{\Vert \langle \mathbf{t} -x \mathbf n \rangle
\Vert_\infty}\,:\, {x \in \mathbb R}\,\} \end{aligned}.  \label{E:defmu}
\end{equation}%
Here the symbol $\left\langle u\right\rangle $ denotes the distance to the
nearest integer when $u$ is a real number; for real vectors, $\mathbf{u}$,
it denotes the application of $\left\langle \cdot \right\rangle $
component-wise. We omit the writing of the subscript $S$ when the set $S$ is
clear.

With this notation the \textit{angular Kronecker constant} of $S$ is%
\begin{equation*}
\alpha (S)=\sup \{\mu (\mathbf{t}):\mathbf{t}\in \mathbb{R}^{d}\}
\end{equation*}%
and the \textit{binary Kronecker constant} is 
\begin{equation*}
\beta (S)=\sup \{\mu (\mathbf{t}):\mathbf{t}\in \mathbb{\{}0,1/2\}^{d}\}.
\end{equation*}%
As noted in \cite{HR1}, periodicity allows one to limit $x$ to an interval
of length $1$ and the vectors $\mathbf{k}$ to a finite set. Consequently,
inf and sup can be replaced by min and max. A choice of $x$ which minimizes $%
\mu _{S}(\mathbf{t})$ is known as a \textit{best approximate} for $\mathbf{t}
$ relative to $S$.

\section{Upper bounds on the Kronecker constants}

\subsection{Setting up a `greedy' algorithm.}

Throughout the remainder of the paper we assume $a<b<n$ are fixed with $\gcd
(a,b)=1,$ and $n$ is suitably large. When we write $\mu (t_{1},t_{2},t_{3})$
we mean $\mu _{\{a,b,n\}}(t_{1},t_{2},t_{3})$. Recall that $E_{n}$ denotes
the value claimed by Theorem \ref{T:three} for $\alpha (a,b,n)$. It is known
that for all positive integers $a,b$, 
\begin{equation*}
\alpha (a,b,n)\rightarrow \alpha (a,b)=1/(2a+2b)\text{ as }n\rightarrow
\infty .
\end{equation*}%
It is clear that $\lim_{n\rightarrow \infty }E_{n}=$ $1/(2a+2b)$, thus by
taking $n$ sufficiently large we can assume without loss of generality that
both $\alpha (a,b,n)$ and $E_{n}$ are smaller than $1/(a+b)$.

Our strategy will be to start with a best approximate for $(t_{1},t_{2})$,
relative to $\{a,b\}$, and then to modify it slightly to improve the
approximation of $t_{3}$ relative to $n$. In many cases, this suffices. In
other cases a related, `second best' approximate must be suitably modified.
These ideas will be made precise in this section.

In \cite{HR1} it is shown that for any pair of real numbers, $(t_{1},t_{2}),$
there are always a real number $x$ and integers $k_{1},k_{2}$ with the
property 
\begin{equation}
ax-(t_{1}+k_{1})=-(bx-(t_{2}+k_{2})).  \label{oppsign}
\end{equation}%
In fact, there is always a triple, $x,k_{1},k_{2},$ that not only satisfies (%
\ref{oppsign}), but is also a best approximate to $(t_{1},t_{2}),$ by which
we mean that 
\begin{equation*}
\mu _{\{a,b\}}(t_{1},t_{2})=\left\Vert
x(a,b)-(t_{1}+k_{1},t_{2}+k_{2})\right\Vert _{\infty }.
\end{equation*}

Given an $x,$ $k_{1},k_{2}$ satisfying (\ref{oppsign}), we will put 
\begin{equation*}
\lambda _{x}=\left\vert ax-(t_{1}+k_{1})\right\vert .
\end{equation*}%
If it is the case that $\lambda _{x}\leq (b-a)/(2n)$, then choose $z$ such
that $\left\vert nz-nx\right\vert \leq 1$ and $nz\equiv t_{3}\mod 1$, say $%
nz=t_{3}+k_{3}$ for integer $k_{3}$. An easy calculation gives 
\begin{eqnarray*}
\left\vert az-(t_{1}+k_{1})\right\vert &\leq &\left\vert a(z-x)\right\vert
+\left\vert ax-(t_{1}+k_{1})\right\vert \leq \frac{a}{n}+\lambda _{x}\leq 
\frac{b+a}{2n}, \\
\left\vert bz-(t_{2}+k_{2})\right\vert &\leq &\left\vert b(z-x)\right\vert
+\left\vert bx-(t_{2}+k_{2})\right\vert \leq \frac{b}{n}+\lambda _{x}\leq 
\frac{3b-a}{2n},
\end{eqnarray*}%
and 
\begin{equation*}
\left\vert nz-(t_{3}+k_{3})\right\vert =0.
\end{equation*}%
This proves

\begin{lemma}
\label{smalllambda}If $\lambda _{x}\leq (b-a)/(2n)$, then%
\begin{equation*}
\mu (t_{1},t_{2},t_{3})\leq \left\Vert
z(a,b,n)-(t_{1}+k_{1},t_{2}+k_{2},t_{3}+k_{3})\right\Vert _{\infty }\leq 
\frac{3b-a}{2n}.
\end{equation*}
\end{lemma}

For sufficiently large $n$, $(3b-a)/(2n)\leq E_{n}$, thus our interest is
(primarily) with those $x$ such that $\lambda _{x}>(b-a)/(2n)$.

Next, we will explain what we mean by \textquotedblleft modify it slightly
to improve the approximation of $t_{3}$ relative to $n$\textquotedblright .
Take any a real number $E\geq \lambda _{x}$, and suppose real number $z$ and
integer $k_{3}$ have been chosen satisfying 
\begin{equation*}
nz=t_{3}+k_{3}\text{ and }\left\vert nz-nx\right\vert \leq 1.
\end{equation*}%
If $\left\vert nz-nx\right\vert \leq \lambda _{x}$ we will not modify $x$.
Otherwise, our strategy will be to replace $x$ by a nearby point, $x\pm
\delta $, to get a better approximate. This strategy will culminate with the
bounds of Corollaries \ref{strategy} and \ref{small}.

Case 1: \thinspace $ax-(t_{1}+k_{1})\geq 0$. First, suppose $z\leq x$. We
will replace $x$ by $x-\delta $ for a suitably small $\delta >0$. This will
have the effect of bringing down the size of the third component of%
\begin{equation*}
\left\Vert (x-\delta
)(a,b,n)-(t_{1}+k_{1},t_{2}+k_{2},t_{3}+k_{3})\right\Vert _{\infty }.
\end{equation*}%
However, we will pay a cost for this: the size of the second component will
increase. We will permit only enough of an increase to balance these
approximations of $t_{2}$ and $t_{3}$. Thus $0<\delta <x-z$ will be chosen
so that%
\begin{equation*}
\left\vert (x-\delta )n-(t_{3}+k_{3})\right\vert =\left\vert (x-\delta
)b-(t_{2}+k_{2})\right\vert .
\end{equation*}%
Recalling that $nz=t_{3}+k_{3}$ and $xb-(t_{2}+k_{2})=-\lambda _{x}$, this
gives 
\begin{equation}
\delta =\frac{\left\vert nx-nz\right\vert -\lambda _{x}}{b+n}.  \label{delta}
\end{equation}

Since $ax-(t_{1}+k_{1})=\lambda _{x}>0$, it is clear that for small $\delta
>0,$ 
\begin{equation*}
\left\vert (x-\delta )a-(t_{1}+k_{1})\right\vert \leq
ax-(t_{1}+k_{1})=\lambda _{x}=\left\vert xb-(t_{2}+k_{2})\right\vert .
\end{equation*}%
As $\left\vert (x-\delta )a-(t_{1}+k_{1})\right\vert $ changes more slowly
than $\left\vert (x-\delta )b-(t_{2}+k_{2})\right\vert ,$ the inequality
above actually holds for all $\delta >0$, consequently, 
\begin{equation*}
\begin{aligned} &\left\Vert (x-\delta
)(a,b,n)-(t_{1}+k_{1},t_{2}+k_{2},t_{3}+k_{3})\right\Vert _{\infty }
=\left\vert (x-\delta )b-(t_{2}+k_{2})\right\vert \\ &\quad =\left\vert
(x-\delta )n-(t_{3}+k_{3})\right\vert =\lambda _{x}+b\left( \frac{\left\vert
nx-nz\right\vert -\lambda _{x}}{b+n}\right) \label{z<x} \end{aligned}.
\end{equation*}%
This shows that if, in addition to the assumption $z\leq x$, $z$ satisfies
the requirement 
\begin{equation*}
\lambda _{x}+b\left( \frac{\left\vert nx-nz\right\vert -\lambda _{x}}{b+n}%
\right) \leq E,
\end{equation*}%
equivalently,%
\begin{equation}
z\geq x+\frac{n\lambda _{x}-(b+n)E}{bn},  \label{sizez<x}
\end{equation}%
then \ 
\begin{equation}
\left\Vert (x-\delta
)(a,b,n)-(t_{1}+k_{1},t_{2}+k_{2},t_{3}+k_{3})\right\Vert _{\infty }\leq E.
\label{normz<x}
\end{equation}%
Of course, (\ref{normz<x}) implies $\mu (t_{1},t_{2},t_{3})\leq E$.

Otherwise, $z>x$. In this case, we replace $x$ by $x+\delta $, where $\delta
>0$ is chosen to balance the cost of approximating $t_{1}$ and $t_{3}$. This
means we choose $\delta <z-x$ such that 
\begin{equation*}
\left\vert (x+\delta )n-(t_{3}+k_{3})\right\vert =(x+\delta )a-(t_{1}+k_{1}),
\end{equation*}%
in other words, 
\begin{equation*}
\delta =\frac{\left\vert nx-nz\right\vert -\lambda _{x}}{a+n}.
\end{equation*}%
If $(x+\delta )b\leq t_{2}+k_{2}$, then clearly 
\begin{equation*}
\left\vert (x+\delta )b-(t_{2}+k_{2})\right\vert \leq \left\vert
xb-(t_{2}+k_{2})\right\vert \leq (x+\delta )a-(t_{1}+k_{1}).
\end{equation*}%
But even when $(x+\delta )b>t_{2}+k_{2}$, we will still have the bound 
\begin{equation*}
\left\vert (x+\delta )b-(t_{2}+k_{2})\right\vert \leq (x+\delta
)a-(t_{1}+k_{1}),
\end{equation*}%
provided%
\begin{equation*}
\delta \leq \frac{xa-(t_{1}+k_{1})-(xb-(t_{2}+k_{2}))}{b-a}=\frac{2\lambda
_{x}}{b-a}\text{.}
\end{equation*}%
But $\delta \leq z-x\leq 1/n$, and by assumption $1/n<2\lambda _{x}/(b-a)$,
hence this condition is automatically satisfied. Therefore 
\begin{eqnarray*}
\left\Vert (x+\delta
)(a,b,n)-(t_{1}+k_{1},t_{2}+k_{2},t_{3}+k_{3})\right\Vert _{\infty }
&=&\left\vert (x+\delta )a-(t_{1}+k_{1})\right\vert \\
&=&\lambda _{x}+a\left( \frac{\left\vert nx-nz\right\vert -\lambda _{x}}{a+n}%
\right) .
\end{eqnarray*}%
The same reasoning as in the first case shows that if, in addition to the
assumption that $z>x$,%
\begin{equation}
z\leq x+\frac{(a+n)E-n\lambda _{x}}{an},  \label{sizez>x}
\end{equation}%
then 
\begin{equation}
\left\Vert (x-\delta
)(a,b,n)-(t_{1}+k_{1},t_{2}+k_{2},t_{3}+k_{3})\right\Vert _{\infty }\leq E.
\label{normz>x}
\end{equation}%
Thus again $\mu (t_{1},t_{2},t_{3})\leq E$.\smallskip

Case 2: $ax-(t_{1}+k_{1})<0$. Then $\lambda _{x}=$ $-\left(
ax-(t_{1}+k_{1})\right) =bx-(t_{2}+k_{2})$.

If $z\leq x,$ we choose $0<$ $\delta <x-z$ so that \ $\left\vert (x-\delta
)n-(t_{3}+k_{3})\right\vert =\left\vert (x-\delta
)a-(t_{1}+k_{1})\right\vert $. The assumption $\lambda _{x}\geq (b-a)/(2n)$
\ ensures that also 
\begin{equation*}
\left\vert (x-\delta )b-(t_{2}+k_{2})\right\vert \leq \left\vert (x-\delta
)a-(t_{1}+k_{1})\right\vert ,
\end{equation*}%
thus it follows that if 
\begin{equation}
x\geq z\geq x+\frac{n\lambda _{x}-(a+n)E}{an},  \label{sizebz<x}
\end{equation}%
then 
\begin{equation}
\left\Vert (x-\delta
)(a,b,n)-(t_{1}+k_{1},t_{2}+k_{2},t_{3}+k_{3})\right\Vert _{\infty }\leq E.
\label{normbz<x}
\end{equation}

If $z\geq x$, we choose $0<$ $\delta <z-x$ so that \ $\left\vert (x+\delta
)n-(t_{3}+k_{3})\right\vert =\left\vert (x+\delta
)b-(t_{2}+k_{2})\right\vert $, and again one can check that if%
\begin{equation}
x\leq z\leq x+\frac{(b+n)E-n\lambda _{x}}{bn},  \label{sizebz>x}
\end{equation}%
then 
\begin{equation}
\left\Vert (x+\delta
)(a,b,n)-(t_{1}+k_{1},t_{2}+k_{2},t_{3}+k_{3})\right\Vert _{\infty }\leq E.
\label{normbz>x}
\end{equation}

These are the key ideas behind the next lemma.

\begin{lemma}
\label{mainprelim}Suppose there exist $x\in \mathbb{R}$ and integers $%
k_{1},k_{2}$ such that 
\begin{equation*}
ax-(t_{1}+k_{1})=-(bx-(t_{2}+k_{2})).
\end{equation*}%
Let $\left\vert ax-(t_{1}+k_{1})\right\vert =\lambda _{x}$ and for $E\geq
\lambda _{x}$, put 
\begin{equation*}
z_{1}(E,x)=x+\frac{n\lambda _{x}-(b+n)E}{bn}\text{, }z_{2}(E,x)=x+\frac{%
(a+n)E-n\lambda _{x}}{an},
\end{equation*}%
\begin{equation*}
z_{3}(E,x)=x+\frac{n\lambda _{x}-(a+n)E}{an}\text{, }z_{4}(E,x)=x+\frac{%
(b+n)E-n\lambda _{x}}{bn}.
\end{equation*}%
Assume $\lambda _{x}>(b-a)/(2n)$. Then $z_{1}(E,x)\leq z_{2}(E,x)$ and $%
z_{3}(E,x)\leq z_{4}(E,x)$. Furthermore,

(i) If $ax-(t_{1}+k_{1})=\lambda _{x}$ and there exists $z\in \lbrack
z_{1}(E,x),z_{2}(E,x)]$ such that $nz\equiv t_{3}\mod 1$, then $\mu
(t_{1},t_{2},t_{3})\leq E$.

(ii) If $-\left( ax-(t_{1}+k_{1})\right) =\lambda _{x}$ and there exists $%
z\in \lbrack z_{3}(E,x),z_{4}(E,x)]$ such that $nz\equiv t_{3}\mod 1$, then $%
\mu (t_{1},t_{2},t_{3})\leq E$.
\end{lemma}

\begin{proof}
As $E\geq \lambda _{x}$, we have $z_{1}(E,x)\leq z_{2}(E,x)$ and $%
z_{3}(E,x)\leq z_{4}(E,x)$.

We note that if there is some $z\in \lbrack z_{1}(E,x),z_{2}(E,x)]$ (or $%
z\in \lbrack z_{3}(E,x),z_{4}(E,x)])$ such that $nz\equiv t_{3}\mod 1$, then
there is a possibly different choice of $z,$ belonging to the same interval,
still satisfying $nz\equiv t_{3}\mod 1,$ and having the additional property
that $\left\vert nx-nz\right\vert \leq 1$.\ We will work with such a $z$.

First, suppose $ax-(t_{1}+k_{1})=\lambda _{x}$. If $z_{1}(E,x)\leq z\leq x$,
then it follows from (\ref{sizez<x}) and (\ref{normz<x}) (with the $\delta $
described in (\ref{delta})) that 
\begin{equation*}
\mu (t_{1},t_{2},t_{3})\leq \left\Vert (x-\delta
)(a,b,n)-(t_{1}+k_{1},t_{2}+k_{2},t_{3}+k_{3})\right\Vert _{\infty }\leq E.
\end{equation*}

If, instead, $x\leq z\leq z_{2}(E,x)$, we appeal to (\ref{sizez>x}) and (\ref%
{normz>x}) to deduce that 
\begin{equation*}
\mu (t_{1},t_{2},t_{3})\leq \left\Vert (x+\delta
)(a,b,n)-(t_{1}+k_{1},t_{2}+k_{2},t_{3}+k_{3})\right\Vert _{\infty }\leq E.
\end{equation*}

If $-(ax-(t_{1}+k_{1}))=\lambda _{x}$, we similarly call upon (\ref{sizebz<x}%
), (\ref{normbz<x}), (\ref{sizebz>x}) and (\ref{normbz>x}).
\end{proof}

\begin{remark}
\label{modified}Observe that the proof shows that if $E=\mu
(t_{1},t_{2},t_{3})$, then one of $x\pm \delta $ is a best approximate to $%
(t_{1},t_{2},t_{3})$ relative to $\{a,b,n\}$.
\end{remark}

\begin{corollary}
\label{strategy}For all $t_{1},t_{2},t_{3}$ we have 
\begin{equation*}
\mu (t_{1},t_{2},t_{3})\leq \max \left( \frac{n(a+b)\mu
_{\{a,b\}}(t_{1},t_{2})+ab}{2ab+an+bn},\frac{3b-a}{2n}\right) .
\end{equation*}
\end{corollary}

\begin{proof}
Choose $x$ and integers $k_{1},k_{2}$ such that $\mu
_{\{a,b\}}(t_{1},t_{2})=\left\vert ax-(t_{1}+k_{1})\right\vert =\left\vert
bx-(t_{2}+k_{2})\right\vert =\lambda _{x}$. As shown in Lemma \ref%
{smalllambda}, if $\lambda _{x}\leq (b-a)/(2n)$, then 
\begin{equation*}
\mu (t_{1},t_{2},t_{3})\leq (3b-a)/(2n),
\end{equation*}%
so assume otherwise. Set 
\begin{equation*}
E=\frac{n(a+b)\lambda _{x}+ab}{2ab+an+bn}=\lambda _{x}+\frac{ab(1-2\lambda
_{x})}{2ab+an+bn}.
\end{equation*}%
As $\mu _{\{a,b\}}(t_{1},t_{2})<1/2$, we have $E>\lambda _{x}$, so we may
apply Lemma \ref{mainprelim}. In particular, the intervals $%
[z_{1}(E,x),z_{2}(E,x)]$ and $[z_{3}(E,x),z_{4}(E,x)]$ of Lemma \ref%
{mainprelim} are well defined and they have the same length:%
\begin{equation*}
\frac{(a+n)E-n\lambda _{x}}{an}+\frac{(b+n)E-n\lambda _{x}}{bn}=E\left( 
\frac{2ab+an+bn}{abn}\right) -\lambda _{x}\left( \frac{a+b}{ab}\right) =%
\frac{1}{n}\text{.}
\end{equation*}%
Hence both intervals contain some $z$ with $nz\equiv t_{3}\mod 1.$ Now apply
the appropriate part of Lemma \ref{mainprelim}.
\end{proof}

\begin{remark}
\label{length}We point out that the calculations above show that if for some 
$\lambda _{x}\in (0,1/2)$ we let%
\begin{equation*}
E=\frac{n(a+b)\lambda _{x}+ab}{2ab+an+bn}
\end{equation*}%
and $z_{j}=z_{j}(E,x)$ for $j=1,2,3,4$, then the intervals $[z_{1},z_{2}]$
and $[z_{3},z_{4}]$ both have length $1/n$.
\end{remark}

Another important quantity for us will be 
\begin{equation}
L_{n}=\frac{n+ab}{2(an+bn+ab)}.  \label{Ln}
\end{equation}%
When $n\equiv a^{2}\mod(a+b)$, then $L_{n}=E_{n}$. \ For all $n$, $L_{n}\leq
E_{n}$ and clearly $L_{n}>1/(2a+2b)$. The significance of $L_{n}$ is that if 
$\lambda =$ $1/(a+b)-L_{n}$, then%
\begin{equation*}
L_{n}=\frac{n(a+b)\lambda +ab}{2ab+an+bn}.
\end{equation*}%
Thus a consequence of the previous corollary is

\begin{corollary}
\label{small}For $n$ large enough, if $\mu _{\{a,b\}}(t_{1},t_{2})\leq
1/(a+b)-L_{n}$, then $\mu (t_{1},t_{2},t_{3})\leq L_{n}$.
\end{corollary}

\begin{proof}
Observe that for large enough $n$, $(3b-a)/2n\leq L_{n}$ and apply the
previous corollary.
\end{proof}

\subsection{The `second best' point.}

Since $L_{n}\leq E_{n}$, the greedy algorthim, as described in the previous
subsection, establishes that $\mu (t_{1},t_{2},t_{3})\leq E_{n}$ for any
pair $(t_{1},t_{2})$ such that $\mu _{\{a,b\}}(t_{1},t_{2})\leq
1/(a+b)-L_{n} $. For other $(t_{1},t_{2})$ we will make use of a `second
best' point, which can also be naturally constructed, as explained in the
next lemma.

\begin{lemma}
\label{secondbest} There are a real number $x^{\prime }$ and integers $%
k_{1}^{\prime },k_{2}^{\prime }$ such that $ax^{\prime
}-(t_{1}+k_{1}^{\prime })=-(bx^{\prime }-(t_{2}+k_{2}^{\prime }))$ and%
\begin{equation*}
\left\Vert x^{\prime }(a,b)-(t_{1}+k_{1}^{\prime },t_{2}+k_{2}^{\prime
})\right\Vert _{\infty }=\frac{1}{a+b}-\mu _{\{a,b\}}(t_{1},t_{2}).
\end{equation*}
\end{lemma}

\begin{proof}
Choose a best approximate $x$ and integers $k_{1},k_{2}$ such that 
\begin{equation*}
\left\Vert x(a,b)-(t_{1}+k_{1},t_{2}+k_{2})\right\Vert _{\infty }=\mu
_{\{a,b\}}(t_{1},t_{2}).
\end{equation*}%
In particular, $ax-(t_{1}+k_{1})=-(bx-(t_{2}+k_{2}))$, consequently, 
\begin{equation*}
x=\frac{t_{1}+k_{1}+t_{2}+k_{2}}{a+b}
\end{equation*}%
and 
\begin{equation}
\left\vert ax-(t_{1}+k_{1})\right\vert =\left\vert \frac{%
at_{2}-bt_{1}-(bk_{1}-ak_{2})}{a+b}\right\vert .  \label{M1}
\end{equation}

This quantity is minimized when we make either the choice $%
bk_{1}-ak_{2}=\lfloor at_{2}-bt_{1}\rfloor $ or the choice $%
bk_{1}-ak_{2}=\lfloor at_{2}-bt_{1}\rfloor $ $+1$, depending on which gives
the lesser value. (Of course, we can find such integers $k_{1},k_{2}$ in
either case because $a,b$ are coprime.) Without loss of generality, assume
the choice $bk_{1}-ak_{2}=\lfloor at_{2}-bt_{1}\rfloor $ gives the minimal
answer (the other case is symmetric). That means%
\begin{equation*}
\mu _{\{a,b\}}(t_{1},t_{2})=ax-(t_{1}+k_{1})=\frac{%
at_{2}-bt_{1}-[at_{2}-bt_{1}]}{a+b}.
\end{equation*}%
Choose integers $g,h$ such that $ag-bh=1$ and put%
\begin{equation}
k_{1}^{\prime }=k_{1}-h\text{ and }k_{2}^{\prime }=k_{2}-g.  \label{k'}
\end{equation}%
Then $bk_{1}^{\prime }-ak_{2}^{\prime }=\lfloor at_{2}-bt_{1}\rfloor $ $+1$ $%
\ $and if we take%
\begin{equation}
x^{\prime }=\frac{t_{1}+k_{1}^{\prime }+t_{2}+k_{2}^{\prime }}{a+b}
\label{x'}
\end{equation}%
then $ax^{\prime }-(t_{1}+k_{1}^{\prime })=-(bx^{\prime
}-(t_{2}+k_{2}^{\prime }))<0$ and%
\begin{eqnarray}
\left\Vert x^{\prime }(a,b)-(t_{1}+k_{1}^{\prime },t_{2}+k_{2}^{\prime
})\right\Vert _{\infty } &=&-\left( ax^{\prime }-(t_{1}+k_{1}^{\prime
})\right)  \notag \\
&=&-\left( \frac{at_{2}-bt_{1}-(\lfloor at_{2}-bt_{1}\rfloor +1)}{a+b}\right)
\notag \\
&=&\frac{1}{a+b}-\mu _{\{a,b\}}(t_{1},t_{2}).  \label{M3}
\end{eqnarray}
\end{proof}

\begin{remark}
\ \label{diff}By construction $x-x^{\prime }=(g+h)/(a+b)$. The reader should
observe that Lemma \ref{mainprelim} applies to this $x^{\prime }$, as well
as the best approximate $x$.
\end{remark}

\begin{lemma}
\label{main}Choose $n$ so large that $1/(a+b)-L_{n}>(b-a)/2n$. Suppose $\mu
_{\{a,b\}}(t_{1},t_{2})=\left\Vert
x(a,b)-(t_{1}+k_{1},t_{2}+k_{2})\right\Vert _{\infty }:=\lambda _{x}$.
Choose integers $g,h$ such that $ag-bh=1$ and define $x^{\prime
},k_{1}^{\prime },k_{2}^{\prime }$ as in (\ref{k'}) and (\ref{x'}). Assume 
\begin{equation*}
\lambda _{x}=ax-(t_{1}+k_{1})>\frac{1}{a+b}-L_{n}.
\end{equation*}%
Fix $E\geq L_{n}$ and define $z_{1}=z_{1}(E,x)$ and $z_{2}=z_{2}(E,x)$ as in
Lemma \ref{mainprelim}. Put%
\begin{equation*}
z_{3}=x^{\prime }+\frac{n\left( \frac{1}{a+b}-\lambda _{x}\right) -(a+n)E}{an%
},\;\;z_{4}=x^{\prime }+\frac{(b+n)E-n\left( \frac{1}{a+b}-\lambda
_{x}\right) }{bn}.
\end{equation*}%
If there is some $z$ with $nz\equiv t_{3}\mod 1$, satisfying $z_{1}\leq
z\leq z_{2}$ or $z_{3}\leq z\leq z_{4}$, then $\mu (t_{1},t_{2},t_{3})\leq
E. $
\end{lemma}

\begin{proof}
Put $\lambda _{x^{\prime }}:=\left\vert ax^{\prime }-(t_{1}+k_{1}^{\prime
})\right\vert =-\left( ax^{\prime }-(t_{1}+k_{1}^{\prime })\right) $. As
shown in (\ref{M3}), $\lambda _{x^{\prime }}=1/(a+b)-\lambda _{x}$, and
since $\lambda _{x}=\mu _{\{a,b\}}(t_{1},t_{2})\leq 1/(2(a+b))$, it follows
that 
\begin{equation*}
\frac{b-a}{2n}<\lambda _{x}\leq \lambda _{x^{\prime }}<L_{n}\leq E.
\end{equation*}%
In the notation of Lemma \ref{mainprelim}, $z_{3}=z_{3}(E,x^{\prime })$ and $%
z_{4}=z_{4}(E,x^{\prime })$, hence a direct application of that lemma yields
the result.
\end{proof}

\subsection{Conclusion of the proof of the upper bound in Theorem \protect
\ref{T:three}.}

We remind the reader that the numbers $E_{n}$ are defined to be the right
hand side of the formulas given in Theorem \ref{T:three}. Consider any $%
(t_{1},t_{2},t_{3})$ and choose $x\in \mathbb{R}$ and integers $k_{1},k_{2}$
such that 
\begin{equation*}
\mu _{\{a,b\}}(t_{1},t_{2})=\left\Vert
x(a,b)-(t_{1}+k_{1},t_{2}+k_{2})\right\Vert _{\infty }.
\end{equation*}%
Without loss of generality we can assume $\mu
_{\{a,b\}}(t_{1},t_{2})=ax-(t_{1}+k_{1})=\lambda _{x}$ (rather than $%
-(ax-(t_{1}+k_{1}))$), for otherwise replace $\mathbf{t}=(t_{1},t_{2},t_{3})$
by $(-t_{1},-t_{2},-t_{3}),$ noting that $\mu (\mathbf{t})=\mu (-\mathbf{t})$%
.

If $\lambda _{x}\leq 1/(a+b)-L_{n}$, then by Corollary \ref{small}, $\mu
(t_{1},t_{2},t_{3})\leq L_{n}\leq E_{n}$.

Hence we can assume $\lambda _{x}>1/(a+b)-L_{n}$ and we define $x^{\prime }$%
, $z_{1},z_{2},z_{3}$ and $z_{4}$ as in Lemma \ref{main}, with $E_{n}$
taking the role of $E$.

When $R\neq a$, the strategy of the proof is to check that $%
z_{2}-z_{1}+z_{4}-z_{3}\geq 1/n$ and that either $z_{2}-z_{3}$ or $%
z_{1}-z_{4}$ is equal to $j/n$ for some integer $j$. Thus there is always a
choice of $z$ with $nz\equiv t_{3}\mod 1$ and satisfying $z\in \lbrack
z_{1},z_{2}]\cup \lbrack z_{3},z_{4}]$. Appealing to Lemma \ref{main} shows
that $\mu (t_{1},t_{2},t_{3})\leq E_{n}$.

We will outline the details for $R<a$ \ and leave the other $R\neq a$ cases
for the reader. First, note that, regardless of the choice of $R,$ 
\begin{equation}
z_{2}-z_{3}=x-x^{\prime }+\dfrac{2(a+n)E_{n}}{an}-\dfrac{1}{a(a+b)}.
\label{z2-z3}
\end{equation}%
As observed in Remark \ref{diff}, $x-x^{\prime }=(g+h)/(a+b),$ where $g$ and 
$h$ are integers such that $ag-bh=1$. As $a(g+h)=1+h(a+b)\equiv 1\mod (a+b)$%
, there is an integer $v$ such that $T=g+h+v(a+b)$.

Choosing integers $M$ and $M^{\prime }$ such that $n=M(a+b)+r$ and $%
R=M^{\prime }(a+b)+rT,$ and substituting in the value of $E_{n}$ for $R<a$
gives the identity%
\begin{eqnarray*}
z_{2}-z_{3}-\frac{1}{n} &=&\dfrac{g+h}{a+b}-\dfrac{1}{a(a+b)}+\frac{%
n+a(a+b)-aR}{an(a+b)}-\frac{1}{n} \\
&=&\frac{M(g+h)-M^{\prime }-rv}{n}.
\end{eqnarray*}%
This shows there is an integer $j=$ $M(g+h)-M^{\prime }-rv$ such that $%
z_{2}=z_{3}+j/n$.

A straight forward calculation gives%
\begin{equation*}
z_{2}-z_{1}+z_{4}-z_{3}=\frac{2(an+bn+2ab)E_{n}-n}{abn}.
\end{equation*}
It follows that $z_{2}-z_{1}+z_{4}-z_{3}\geq 1/n$ if and only if $E_{n}\geq
(n+ab)/(2(an+bn+2ab))$ and this latter condition is certainly true.

Now suppose that $R=a$. Again, using (\ref{z2-z3}), taking the value of $%
E_{n}$ for $R=a,$ and simplifying gives 
\begin{equation*}
z_{2}-z_{3}-\dfrac{2aE_{n}}{n(a+b)}=\dfrac{g+h}{a+b}+\dfrac{n+ab}{an(a+b)}-%
\dfrac{1}{a(a+b)}=\dfrac{n(g+h)+b}{n(a+b)}.
\end{equation*}%
Since $a^{2}T\equiv a=R\equiv rT$ and $T$ is relatively prime to $a+b$, we
have 
\begin{equation*}
n\equiv r\equiv a^{2}\mod (a+b).
\end{equation*}

Hence there is an integer $M^{\prime \prime }$ such that $n=M^{\prime \prime
}(a+b)+a^{2}$. Using again the identity $a(g+h)=1+h(a+b)$ gives 
\begin{equation*}
\dfrac{n(g+h)+b}{n(a+b)}=\dfrac{M^{\prime \prime }(g+h)}{n}+\dfrac{%
a^{2}(g+h)+b}{n(a+b)}=\dfrac{M^{\prime \prime }g+M^{\prime \prime }h+ah+1}{n}%
.
\end{equation*}%
Thus 
\begin{equation*}
z_{2}=z_{3}+\dfrac{2aE_{n}}{n(a+b)}+\dfrac{s}{n}
\end{equation*}%
for some integer $s$. In particular, $z_{2}>z_{3}+s/n$.

We argue next that $z_{2}\leq z_{4}+s/n$, which is equivalent to proving
that $z_{4}-z_{3}\geq 2aE_{n}/(na+nb)$. To see this, note that the
definition of $z_{3}$ and $z_{4}$ gives 
\begin{equation*}
z_{4}-z_{3}=\dfrac{(b+n)E_{n}}{bn}+\dfrac{(a+n)E_{n}}{an}-\dfrac{\frac{1}{a+b%
}-\lambda _{x}}{a}-\dfrac{\frac{1}{a+b}-\lambda _{x}}{b}.
\end{equation*}%
Because $1/(a+b)-\lambda _{x}<L_{n}=E_{n}$ we have 
\begin{equation*}
\begin{aligned} z_4-z_3-\dfrac{2aE_n}{n(a+b)} &> E_n \cdot
\left(\dfrac{b+n}{bn}+\dfrac{a+n}{an}-\dfrac{2a}{n(a+b)} \right)- E_n\cdot
\left(\dfrac 1 a +\dfrac 1 b \right)\\ &=\dfrac{2bE_n}{(a + b) n }>0
\end{aligned}.
\end{equation*}%
Therefore $[z_{1},z_{4}+s/n]\subset \lbrack z_{1},z_{2}]\cup \lbrack
z_{3}+s/n,z_{4}+s/n]$. The length, $V,$ of $[z_{1},z_{4}+s/n]$ is the sum of
the lengths of the intervals $[z_{1},z_{2}]$ and $[z_{3}+s/n,z_{4}+s/n]$,
but with the length of the overlap, the subinterval $[z_{3}+s/n,z_{2}]$,
subtracted: 
\begin{equation*}
\begin{aligned} V&=z_4-z_3+z_2-z_1-\dfrac{2aE_n}{n(a+b)} \\ &
=\dfrac{2(an+bn+2ab)E_n-n}{abn}-\dfrac{2aE_n}{n(a+b)} \\ &=\dfrac 1 n +
\dfrac{b (a b + n)}{(a + b) n (a b + a n + b n)}>\dfrac 1 n. \end{aligned}
\end{equation*}%
We can again conclude that there is an integer $z\in \lbrack
z_{1},z_{2}]\cup \lbrack z_{3},z_{4}]$, with $nz\equiv t_{3}\mod1$.

\section{Lower bounds on the Kronecker constants}

Continue with the standard assumptions: $a$, $b$ and $n$ are positive
integers with $a<b<n$, $\gcd (a,b)=1$ and $n$ sufficiently large. Recall
that $E_{n}$ is the value claimed by Theorem \ref{T:three} for $\alpha
(a,b,n)$.

For $R\neq a$, the proof that $\alpha (a,b,n)$ $\geq E_{n}$ will follow from
showing that $E_{n}$ is equal to the binary Kronecker constant, $\beta
(a,b,n),$ since it is obvious that $\alpha (a,b,n)\geq \beta (a,b,n)$ for
all $a,b,n$. Unfortunately, when $R=a$, we will see that $\beta
(a,b,n)<E_{n}.$ Thus a different and more direct argument will be given to
establish the specified lower bound in that case. These arguments are very
technical.

The key to obtaining the needed lower bounds is that, depending on the size
of $\mu _{\{a,b\}}(t_{1},t_{2})$, we can restrict the search for the best
approximate $x$ for the $(t_{1},t_{2},t_{3})$-approximation problem to a
small range of real numbers. The first step towards this is to describe the
uniqueness modulo $1$ of some of the intervals used in the argument.

\begin{lemma}
\label{L:uniqueX} Suppose that for some real $t_{1}$ and $t_{2}$ there are
real numbers $x$ and $y$ and integers $j_{1}$, $j_{2}$, $k_{1}$ and $k_{2}$
such that 
\begin{equation*}
\begin{aligned} ax-t_1-j_1=ay-t_1-k_1=-(bx-t_2-j_2)=-(by-t_2-k_2)
\end{aligned}.
\end{equation*}%
Then there is an integer $s$ such that $%
(y,j_{2},k_{2})=(s+x,as+j_{1},bs+k_{1})$.
\end{lemma}

\begin{proof}
By arguments similar to those used in the proof of Lemma \ref{secondbest}, 
\begin{equation*}
\dfrac{at_{2}-bt_{1}+aj_{2}-bj_{1}}{a+b}=\dfrac{at_{2}-bt_{1}+ak_{2}-bk_{1}}{%
a+b}.
\end{equation*}%
Consequently, $a(k_{2}-j_{2})=b(k_{1}-j_{1})$. Because $\gcd (a,b)=1$ there
is an integer $s$ such that $k_{1}-j_{1}=as$ and thus $k_{2}-j_{2}=bs$. It
follows that 
\begin{equation*}
y-x=\frac{k_{1}-j_{1}+k_{2}-j_{2}}{a+b}=s\text{.}
\end{equation*}
\end{proof}

\subsection{Calculating the binary Kronecker constants.}

As noted in \cite{HR3}, there is a ``toggling trick" for $\{0,1/2\}$ - valued
functions. Suppose $\theta $ is defined on $\{n_{j}\}\subseteq \mathbb{Z}$
by $\theta (n_{j})=\theta _{j}\in \{0,1/2\}$ for all $j$. Then we have $\mu
_{\{n_{j}\}}(\theta )=\mu _{\{n_{j}\}}(\widetilde{\theta })$ where $%
\widetilde{\theta }_{j}=\theta _{j}$ if $n_{j}$ is even and \ $\widetilde{%
\theta }_{j}=1/2-\theta _{j}$ if $n_{j}$ is odd. With $\{n_{1},n_{2}\}=\{a,b%
\}$, since the assumption that $\gcd (a,b)=1$ implies at least one of $a$ or 
$b$ is odd, the four binary possibilities break into pairs, which are
equivalent under toggling. One of these pairs includes $(0,0)$ with $\mu
_{\{a,b\}}(0,0)=0$. So computing $\beta (a,b)$ reduces to computing just one
of the four binary possibilities.

\begin{lemma}
\label{L:binary2element} We have $\beta (a,b)=\alpha (a,b)=1/(2a+2b)$. To be
more precise,%
\begin{equation*}
\beta (a,b)=\left\{ 
\begin{array}{cc}
\mu _{\{a,b\}}(0,1/2)=\mu _{\{a,b\}}(1/2,0) & \text{if }a,b\text{ are both
odd} \\ 
\mu _{\{a,b\}}(1/2,0)=\mu _{\{a,b\}}(1/2,1/2) & \text{if }a\text{ is even, }b%
\text{ odd} \\ 
\mu _{\{a,b\}}(0,1/2)=\mu _{\{a,b\}}(1/2,1/2) & \text{if }a\text{ is odd, }b%
\text{ even}%
\end{array}%
\right. .
\end{equation*}

Furthermore, if $\mu _{\{a,b\}}(\theta _{1},\theta _{2})=0,$ then for all
real $\theta _{3}$, $\mu (\theta _{1},\theta _{2},\theta _{3})\leq
(3b-a)/(2n)$.
\end{lemma}

\begin{proof}
Suppose that $a$ is odd and that $\theta (a)=0=\theta _{1}$, $\theta
(b)=1/2=\theta _{2}$. From \cite{HR1}, we know there are a real number $x$
and integers $k_{1},k_{2}$ such that 
\begin{equation*}
\mu _{\{a,b\}}(\theta _{1},\theta _{2})=\Vert (\theta _{1},\theta
_{2})-x(a,b)+(k_{1},k_{2})\Vert _{\infty }.
\end{equation*}%
As in (\ref{M1}), 
\begin{equation*}
\mu _{\{a,b\}}(0,1/2)=\dfrac{|a\theta _{2}-b\theta _{1}+ak_{2}-bk_{1}|}{a+b}=%
\dfrac{|(2k_{2}+1)a-2k_{1}b|}{2a+2b}.
\end{equation*}%
Because $a$ is odd, it follows that $\beta (a,b)\geq \mu
_{\{a,b\}}(0,1/2)\geq 1/(2a+2b)$. But $\mu _{\{a,b\}}(0,0)=$ $\mu
_{\{a,b\}}(1/2,1/2)=0,$ and we always have, $\mu _{\{a,b\}}(0,1/2)\leq
\alpha (a,b)=1/(2a+2b)$ (\cite{HR1}), hence $\beta (a,b)=$ $\mu
_{\{a,b\}}(0,1/2)=\alpha (a,b)=1/(2a+2b)$.

The case $a$ is even is similar.

The last claim of the lemma follows from Lemma \ref{smalllambda}.
\end{proof}

\textbf{Notation and Elementary Observations:} For the remainder of this
subsection we will set 
\begin{equation*}
(t_{1},t_{2})=\left\{ 
\begin{array}{cc}
(1/2,0) & \text{if }b\text{ is odd} \\ 
(0,1/2) & \text{if }b\text{ is even}%
\end{array}%
\right. .
\end{equation*}%
It is easy to see from the toggling trick and the previous lemma that%
\begin{equation*}
\beta (a,b,n)=\max \{\mu (t_{1},t_{2},t_{3}):t_{3}=0,1/2\}.
\end{equation*}%
To calculate $\mu (t_{1},t_{2},t_{3}),$ we will first exhibit a specific
interval that contains a best approximate to $\mathbf{t}=(t_{1},t_{2},t_{3})$%
. This is done in Lemma \ref{singleInt2}. In Prop. \ref{binary1} we use
ideas from the proof of Lemma \ref{mainprelim} to calculate $\mu
(t_{1},t_{2},t_{3})$ by minimizing $\left\Vert \left\langle \mathbf{t}%
-y(a,b,n)\right\rangle \right\Vert _{\infty }$ over $y$ in this interval.
The final step is to find the maximum value of $\mu (t_{1},t_{2},t_{3})$
over $\mathbf{t}\in \{0,1/2\}^{3}$.

If $b$ is odd, then $\gcd (2a,b)=1,$ so there are integers $G,H$ such that $%
2aG-bH=1$. Notice that $H$ is odd. Set $g=2G$ and $h=H$. Put 
\begin{equation*}
k_{1}=(H-1)/2\text{ and }k_{2}=G.
\end{equation*}%
If, instead, $b$ is even there are integers $H$ and (odd) $G$ such that $%
aG-b2H=1$. Set $g=G$ and $h=2H$, 
\begin{equation*}
k_{1}=H\text{ and }k_{2}=(G-1)/2.
\end{equation*}

In either case, $ag-bh=1$. Put%
\begin{equation}
x=\frac{g+h}{2(a+b)}.  \label{defx}
\end{equation}%
As usual, set $\lambda _{x}:=ax-t_{1}-k_{1}$. The reader should check that $%
\lambda _{x}=-(bx-t_{2}-k_{2})=1/(2a+2b)$. For $z\in \mathbb{R}$, define%
\begin{equation}
y_{z}=\left\{ 
\begin{array}{cc}
x+\frac{{\huge |nz-nx|-\lambda }_{x}}{{\huge a+n}} & \text{if }z\geq x \\ 
x-\frac{{\huge |nz-nx|-\lambda }_{x}}{{\huge b+n}} & \text{if }z<x%
\end{array}%
\right.  \label{defy}
\end{equation}%
Observe that the expression, $y_{z},$ has the form $x\pm \delta $ that
appears in the lead up to the proof of Lemma \ref{mainprelim}.

Since $\lambda _{x}=1/(2(a+b))$, 
\begin{equation}
\frac{n(a+b)\lambda _{x}+ab}{2ab+an+bn}=\frac{n+2ab}{2(an+bn+2ab)}:=E\text{.}
\label{E}
\end{equation}%
Recall that $E$ was an important number in the upper bound argument.

Let $z_{1}=z_{1}(E,x)$ and $z_{2}=z_{2}(E,x)$ be as defined in Lemma \ref%
{mainprelim}. Our next step is to show that we can find a best approximate
to $(t_{1},t_{2},t_{3})$ in $[y_{z_{1}},y_{z_{2}}]$.

\begin{lemma}
\label{singleInt2} Assume the notation is as above. For each choice of $%
t_{3}\in \{0,1/2\}$ there is some $y\in \lbrack y_{z_{1}},y_{z_{2}}]$ and
integer $k_{3}$ such that 
\begin{equation*}
\mu (t_{1},t_{2},t_{3})=\Vert
(t_{1},t_{2},t_{3})-y(a,b,n)+(k_{1},k_{2},k_{3})\Vert _{\infty }.
\end{equation*}
\end{lemma}

\begin{proof}
Here $E>1/(2a+2b)=\lambda _{x}=ax-t_{1}-k_{1}>0$. For $n$ sufficiently large
we have $\lambda _{x}\geq (b-a)/(2n)$, hence we may use Lemma \ref%
{mainprelim}.

The choice of $E$ ensures that $z_{2}-z_{1}=1/n$ (see Remark \ref{length}).
Thus for each real $t_{3}$ there is an integer $k_{3}$ such that $%
z=(t_{3}+k_{3})/n$ belongs to the interval $[z_{1},z_{2}]$. Consequently, an
application of Lemma \ref{mainprelim} implies $\mu (t_{1},t_{2},t_{3})\leq E$%
. Furthermore, the proof of that lemma shows that if $(t_{3}+k_{3})/n$ is $%
z_{1}$ or $z_{2}$, then 
\begin{equation*}
E=\Vert (t_{1},t_{2},t_{3})-y_{z}(a,b,n)+(k_{1},k_{2},k_{3})\Vert _{\infty }.
\end{equation*}

As noted in \cite{HR1}, there is a real number $u$ and an integer vector $%
\mathbf{k}$ such that $\mu (t_{1},t_{2},t_{3})=\Vert
(t_{1},t_{2},t_{3})-u(a,b,n)+\mathbf{k}\Vert _{\infty }$. Because $\mu
(t_{1},t_{2},t_{3})\leq \alpha (a,b,n)<1/2$, the vector $\mathbf{k}=(j,k,m)$
where $j$ is the unique integer nearest to $au-t_{1}$, etc.

Case 1: Suppose that $(t_{1}+j)/a\leq (t_{2}+k)/b$. One can easily check
that there is an integer $H\geq 0$ such that 
\begin{equation*}
\dfrac{t_{2}+k}{b}-\dfrac{t_{1}+j}{a}=\dfrac{\lambda _{x}(a+b)+H}{ab}.
\end{equation*}%
Set $y=(t_{1}+t_{2}+j+k)/(a+b)$ so that 
\begin{equation*}
ay-t_{1}-j=-(by-t_{2}-k)=\dfrac{at_{2}-bt_{1}+ak-bj}{a+b}=\lambda _{x}+%
\dfrac{H}{a+b}.
\end{equation*}%
First, suppose that $H\geq 1$. Since $\lambda _{x}>0$, if $u\geq y$ then 
\begin{equation*}
\mu (t_{1},t_{2},t_{3})\geq \langle au-t_{1}\rangle =au-t_{1}-j\geq
ay-t_{1}-j\geq \dfrac{1}{a+b},
\end{equation*}%
while if $u\leq y$, then 
\begin{equation*}
\mu (t_{1},t_{2},t_{3})\geq \langle bu-t_{2}\rangle =t_{2}+k-bu\geq
t_{2}+k-by\geq \dfrac{1}{a+b}.
\end{equation*}%
However, for $n$ large enough, $\mu (t_{1},t_{2},t_{3})\leq \alpha
(a,b,n)<1/(a+b)$. This contradiction forces $H=0$.

With $H=0$, the real number $y$ meets the hypotheses of Lemma \ref{L:uniqueX}%
, hence there is an integer $s$ such that $(y,j,k)=(x+s,as+k_{1},bs+k_{2})$.

If $u>y_{z_{2}}+s$, then because $y_{z_{2}}>x$ we have $%
u>x+s>s+(t_{1}+j_{1})/a=(t_{1}+j)/a$. \ From the proof of Lemma \ref%
{mainprelim} we see that 
\begin{equation*}
\begin{aligned} \langle au-t_1 \rangle & =au-t_1-j =a(u-s)-t_1-k_1
>ay_{z_2}-t_1-k_1=E \end{aligned}.
\end{equation*}%
But, as observed in the second paragraph, $\mu (t_{1},t_{2},t_{3})\leq E$,
so this gives a contradiction.

Similarly, if $u<y_{z_{1}}+s$, then $u<(t_{2}+k)/b$ and $\left\langle
bu-t_{2}\right\rangle >k_{2}+t_{2}-by_{z_{2}}=E$. Again, this contradicts
the first paragraph.

Therefore $u-s\in \lbrack y_{z_{1}},y_{z_{2}}]$ and it has the same
approximation properties for $(t_{1},t_{2},t_{3})$ as does $u$, hence the
lemma holds.

Case 2: Otherwise, $(t_{1}+j)/a>(t_{2}+k)/b$. In this case, there is an
integer $H\geq 1$ such that 
\begin{equation*}
\dfrac{t_{1}+j}{a}-\dfrac{t_{2}+k}{b}=\dfrac{H-(a+b)\lambda _{x}}{ab}=\dfrac{%
2H-1}{2ab}.
\end{equation*}%
Let $y=(t_{1}+t_{2}+j+k)/(a+b)$. The reader can check that%
\begin{equation*}
by-t_{2}-k=-(ay-t_{1}-j)=\frac{2H-1}{2a+2b}.
\end{equation*}%
If $H\geq 2$ and $u\geq y,$ then $u>(t_{2}+k)/b$ and thus 
\begin{equation*}
\begin{aligned} \langle bu-t_2 \rangle &=|bu-t_2-k|=bu-t_2-k \ge by-t_2-k
\ge \dfrac{3}{2a+2b} \end{aligned}.
\end{equation*}%
If, instead, $u\leq y$, then $u<(t_{1}+j)/a$ and thus 
\begin{equation*}
\begin{aligned} \langle au-t_1 \rangle &=|au-t_1-j|=j+t_1-au \ge j+t_1-ay
\ge \dfrac{3}{2a+2b} \end{aligned}.
\end{equation*}%
However, for $n$ large enough, $\mu (t_{1},t_{2},t_{3})\leq \alpha
(a,b,n)<3/(2a+2b)$, thus $H=1$.

Now we use the fact that $t_{3}\in \{0,1/2\}$. Clearly 
\begin{equation*}
\mu (t_{1},t_{2},t_{3})=\Vert
(-t_{1},-t_{2},-t_{3})-(-u)(a,b,n)+(-j,-k,-m)\Vert _{\infty }.
\end{equation*}%
When $t_{1}=0$ let $j^{\prime }=-j$, otherwise let $j^{\prime }=-j-1$. In
either case, $-t_{1}-j=t_{1}+j^{\prime }$, so $%
-t_{1}-(-u)a-j=t_{1}-(-u)a+j^{\prime }$. Define $k^{\prime }$ and $m^{\prime
}$ similarly. Again, in either case $-t_{2}-k=t_{2}+k^{\prime }$, so $%
(t_{1}+j^{\prime })/a<(t_{2}+k^{\prime })/b$. Moreover, 
\begin{equation*}
\mu (t_{1},t_{2},t_{3})=\Vert (t_{1},t_{2},t_{3})-(-u)(a,b,n)+(j^{\prime
},k^{\prime },m^{\prime })\Vert _{\infty }.
\end{equation*}

Therefore $-u,$ with the integers $j^{\prime }$, $k^{\prime }$ and $%
m^{\prime },$ meets the conditions of the first part of this proof, where
the conclusion of the lemma has already been established.
\end{proof}

With these technical results, we can now calculate $\mu (t_{1},t_{2},t_{3})$.

\begin{proposition}
\label{binary1} Let $n\equiv r\mod (2a+2b)$, with $r\in \lbrack 0,2a+2b)$.
Let $(g+h)r\equiv S\mod (2a+2b)$ with $S\in \lbrack 0,2a+2b)$. Then 
\begin{equation*}
\mu (t_{1},t_{2},0)=%
\begin{cases}
\dfrac{1}{2(a+b)} & \text{if $S=0,1,2a+2b-1$} \\ 
\dfrac{n+bS}{2(a+b)(b+n)} & \text{if }2\leq S\leq 2a \\ 
\dfrac{n+2a^{2}+2ab-aS}{2(a+b)(a+n)} & \text{if $2a<S\leq 2a+2b-2$}%
\end{cases}%
\end{equation*}%
and 
\begin{equation*}
\mu (t_{1},t_{2},1/2)=%
\begin{cases}
\dfrac{1}{2(a+b)} & \text{if $S=a+b,a+b\pm 1$} \\ 
\dfrac{n+a^{2}+ab-aS}{2(a+b)(a+n)} & \text{if $0\leq S<a+b-1$} \\ 
\dfrac{n-ab-b^{2}+bS}{2(a+b)(b+n)} & \text{if $a+b+1<S\leq 3a+b$} \\ 
\dfrac{n+3a^{2}+3ab-aS}{2(a+b)(a+n)} & \text{if $3a+b<S<2a+2b$}%
\end{cases}%
.
\end{equation*}
\end{proposition}

\begin{proof}
To begin, we observe that there are integers $M$ and $N$ such that $%
n=M(2a+2b)+r$ and $r(g+h)=N(2a+2b)+S$. As $x=(g+h)/(2a+2b)$, we have 
\begin{equation*}
nx=M(g+h)+N+S/(2a+2b).
\end{equation*}

Let $E$, $z_{1}$, $z_{2}$, $k_{1},$ $k_{2}$ and $y_{z}$ be as defined in the
preamble to Lemma \ref{singleInt2}. Recall that $z_{1}<x<z_{2}$ and the
choice of $E$ ensures that $z_{2}-z_{1}=1/n$. Also, $%
(t_{1}+k_{1})/a<x<(t_{2}+k_{2})/b$. We note that if $z\leq x$, then $z\leq
y_{z}\leq x$ and, further, that the mapping $z\mapsto y_{z}$ is strictly
decreasing for $z\leq x-1/(2n(a+b))$. Conversely, if $z\geq x$, then $x\leq
y_{z}\leq z$ and the mapping $z\mapsto y_{z}$ is strictly increasing for $%
z\geq x+1/(2n(a+b))$.

For $n$ sufficiently large, we have $(t_{1}+k_{1})/a\leq z_{1}$ and $%
z_{2}\leq (t_{2}+k_{2})/b$. Consequently, for all $u\in \lbrack z_{1},z_{2}]$
(and hence for all $u\in \lbrack y_{z_{1}},y_{z_{2}}]$), 
\begin{equation*}
au-t_{1}-k_{1}\geq 0\text{ and }t_{2}+k_{2}-bu\geq 0.
\end{equation*}%
Moreover, for $u<x$,%
\begin{equation*}
au-t_{1}-k_{1}<ax-t_{1}-k_{1}=\dfrac{1}{2a+2b}=t_{2}+k_{2}-bx<t_{2}+k_{2}-bu.
\end{equation*}%
Likewise, for $u>x$ we have $au-t_{1}-k_{1}>1/(2a+2b)>t_{2}+k_{2}-bu$.

We will assume $t_{3}=1/2$. The case $t_{3}=0$ is similar and will be
discussed briefly at the end of the proof.

By Lemma \ref{singleInt2}, there is a real number $u\in \lbrack
y_{z_{1}},y_{z_{2}}]$ and an integer $k_{3}$ such that 
\begin{equation*}
\mu (t_{1},t_{2},t_{3})=\Vert
(t_{1},t_{2},1/2)-u(a,b,n)+(k_{1},k_{2},k_{3})\Vert _{\infty }.
\end{equation*}%
As proven in \cite{HR1}, at least two of $t_{1}-ua+k_{1}$, $t_{2}-ub+k_{2}$
and $1/2-un+k_{3}$ have opposite signs and absolute values equal to $\mu
(t_{1},t_{2},1/2)$. There are three types to consider, depending on which
pair is opposite in sign and balanced:

Type 1 : $au-t_{1}-k_{1}=$ $t_{2}+k_{2}-ub$ and $\mu
(t_{1},t_{2},t_{3})=au-t_{1}-k_{1}$;

Type 2 : $au-t_{1}-k_{1}=1/2+k_{3}-un$, $\mu
(t_{1},t_{2},t_{3})=au-t_{1}-k_{1}$ and not type 1;

Type 3 : $un-1/2-k_{3}=$ $t_{2}+k_{2}-ub$, $\mu
(t_{1},t_{2},t_{3})=t_{2}-ub+k_{2}$ and not type 1 or 2.

We begin with Type 1. The remarks above imply that $u=x$, $\mu
(t_{1},t_{2},t_{3})=1/(2a+2b)$, and $|1/2+k_{3}-nx|\leq 1/(2a+2b)$. Also, $%
k_{3}$ is the nearest integer to $nx-1/2$. As 
\begin{equation}
nx-1/2=M(g+h)+N+\dfrac{S-a-b}{2a+2b},  \label{nx-1/2}
\end{equation}%
this forces $S$ to be $a+b-1$, $a+b$, or $a+b+1$.

Conversely, if $S$ is one of these three values, then letting $u=x$ and $%
k_{3}=M(g+h)+N$ gives us a Type 1 case.

Next, suppose we are in the Type 2 situation. If $u<x$, then we would have $%
t_{2}+k_{2}-bu>au-t_{1}-k_{1}=\mu (t_{1},t_{2},1/2)$. The contradiction
implies $u\geq x$.

Because $au-t_{1}-k_{1}>0,$ we have $nu-t_{3}-k_{3}<0$. Hence $%
u<(t_{3}+k_{3})/n$ and $k_{3}>nu-t_{3}\geq nx-1/2$, so that $k_{3}\geq
\lceil nx-1/2\rceil $.

Let $z=(t_{3}+k_{3})/n$. If $z\leq x+1/(2n(a+b))$, then 
\begin{equation*}
|nz-nx|=|k_{3}+\frac{1}{2}-nx|\leq \frac{1}{2(a+b)}.
\end{equation*}%
That would imply $S=a+b\pm 1$ or $a+b$ and thus we would be in Type 1.

So $z>x+1/(2n(a+b))$ and because $z>x$ we have $u=x+\delta =y_{z}$, as seen
in the proof of Lemma \ref{mainprelim} (see especially Remark \ref{modified}%
). Since $u\leq y_{z_{2}}$, we know that $z\leq z_{2}<x+1/n$. Also, from the
proof of Lemma \ref{mainprelim}, we have 
\begin{equation}
\begin{aligned} \mu(t_1,t_2,t_3)=au-t_1-k_1 =\dfrac{n
\lambda_x+a|nz-nx|}{a+n} \end{aligned}.  \label{mu}
\end{equation}

Subcase 1: $S\leq a+b$. \ If $k_{3}>\lceil nx-1/2\rceil $, then being an
integer, $k_{3}\geq \lceil nx-1/2\rceil +1$ and thus 
\begin{equation*}
z=(k_{3}+t_{3})/n\geq (nx-1/2+1+1/2)/n=x+1/n>z_{2}.
\end{equation*}%
As this is impossible, $k_{3}=$ $\lceil nx-1/2\rceil $, consequently (\ref%
{nx-1/2}) implies $k_{3}=M(g+h)+N$. Hence 
\begin{equation*}
nz-nx=M(g+h)+N+\frac{1}{2}-[M(g+h)+N+\frac{S}{2(a+b)}]=\dfrac{a+b-S}{2(a+b)}.
\end{equation*}%
In this case, applying (\ref{mu}) and simplifying yields, 
\begin{equation*}
\mu (t_{1},t_{2},t_{3})=\dfrac{n-aS+a(a+b)}{2(a+b)(a+n)}.
\end{equation*}%
Using the formula $nz_{2}-nx=((a+n)E-n\lambda _{x})/a$, allows one to show
that 
\begin{equation*}
\frac{2b-1}{2(a+b)}<n(z_{2}-x)<\frac{2b}{2(a+b)}.
\end{equation*}%
Thus, having $z\leq z_{2}$ is equivalent to $a+b-S<2b$ and hence $a-b<S$.

Subcase 2: $S>a+b$. Again, if $k_{3}>$ $\lceil nx-1/2\rceil $, then 
\begin{equation*}
z=\dfrac{k_{3}+t_{3}}{n}\geq \dfrac{nx-1/2+1+1/2}{n}=x+\dfrac{1}{n}>z_{2}.
\end{equation*}%
So $k_{3}=$ $\lceil nx-1/2\rceil =M(g+h)+N+1$ and therefore 
\begin{equation*}
nz-nx=M(g+h)+N+1+\dfrac{1}{2}-\left[ M(g+h)+N+\dfrac{S}{2a+2b}\right] =%
\dfrac{3a+3b-S}{2a+2b}.
\end{equation*}%
This gives us 
\begin{equation*}
\mu (t_{1},t_{2},t_{3})=\dfrac{\frac{n}{2a+2b}+\frac{a(3a+3b-S)}{2a+2b}}{a+n}%
=\dfrac{n-aS+3a(a+b)}{2(a+b)(a+n)}.
\end{equation*}%
Again having $z\leq z_{2}$ is equivalent to $3a+3b-S<2b$ or $3a+b<S$.

Thus Type 2 implies $0\leq S<a+b-1$ or $3a+b<S<2a+2b$.

Finally, assume we are in Type 3. In this case $u\leq x$ for otherwise $%
au-t_{1}-k_{1}>t_{2}+k_{2}-bu=\mu (t_{1},t_{2},1/2)$. Another consequence of 
$t_{2}+k_{2}-bu>0$ is that we have $t_{3}+k_{3}-nu<0$, so $k_{3}\leq
nu-1/2\leq nx-1/2$.

Letting $z=(t_{3}+k_{3})/n$ gives $z<u\leq x$. If $z\geq x-1/(2n(a+b)),$
then $|nx-k_{3}-t_{3}|$ $\leq 1/(2a+2b)$. But this can happen only in Type 1.

So $z<x-1/(2n(a+b))$. \ Again, the proof of Lemma \ref{mainprelim} implies $%
u=y_{z}$ and 
\begin{equation*}
\mu
(t_{1},t_{2},t_{3})=|y_{z}b-(t_{2}+k_{2})|=|y_{z}n-(t_{3}+k_{3})|=\lambda
_{x}+\dfrac{b(|nx-nz|-\lambda _{x})}{b+n}.
\end{equation*}

Subcase 1: $S\geq a+b$. If $k_{3}<$ $\lfloor nx-1/2\rfloor $, then 
\begin{equation*}
z=\dfrac{k_{3}+t_{3}}{n}\leq \dfrac{nx-1/2+1/2-1}{n}=x-\dfrac{1}{n}<z_{1}%
\text{.}
\end{equation*}%
But this is impossible since $u\geq y_{z_{1}}$ implies $z\geq z_{1}>x-1/n$.
Hence 
\begin{equation*}
nx-nz=M(g+h)+N+\dfrac{S}{2a+2b}-\left[ M(g+h)+N+\dfrac{1}{2}\right] =\dfrac{%
S-a-b}{2a+2b}
\end{equation*}%
and therefore%
\begin{equation*}
\mu (t_{1},t_{2},t_{3})=\dfrac{\frac{n}{2a+2b}+\frac{b(S-a-b)}{2a+2b}}{b+n}=%
\dfrac{n+bS-b(a+b)}{2(a+b)(b+n)}.
\end{equation*}%
The condition that $z\geq z_{1}$ is equivalent to $nx-nz\leq nx-nz_{1}$,
which in turn is equivalent to $S-a-b\leq 2a$, and thus to $S\leq 3a+b$.
Excluding Type 1, we have $a+b+1<S\leq 3a+b$.

Subcase 2: $S<a+b$. Similar arguments to above show that $k_{3}=\left\lfloor
nx-1/2\right\rfloor =M(g+h)+N-1$. Thus 
\begin{equation*}
nx-nz=M(g+h)+N+\dfrac{S}{2a+2b}-\left[ M(g+h)+N-1+\dfrac{1}{2}\right] =%
\dfrac{S+a+b}{2a+2b}.
\end{equation*}%
Here, having $z\geq z_{1}$ is equivalent to $S+a+b\leq 2a$ and thus $S\leq
a-b$, which is not possible because $S\geq 0$.

These three types exhaust all possibilities and are mutually exclusive.
Therefore, $S$ in the appropriate categories implies the desired result.

The arguments are similar when $t_{3}=0$: Type 1 arises when $S=0,1$ or $%
2a+2b-1$, Type 2 when $2a<S\leq 2a+2b-2$ and Type 3 when $2\leq S\leq 2a.$
\end{proof}

\begin{corollary}
\label{binary}(i) If $n\not\equiv a^{2}\mod (a+b)$, then $\beta
(a,b,n)=E_{n} $.

(ii) If $n\equiv a^{2}\mod (a+b)$, then 
\begin{equation*}
\beta (a,b,n)=\frac{n+ab}{2(a+b)(a+n)}<E_{n}\text{.}
\end{equation*}
\end{corollary}

\begin{proof}
This is just a matter of checking which is greater, $\mu (t_{1},t_{2},0)$ or 
$\mu (t_{1},t_{2},1/2)$, in each case. We leave the details for the reader.
\end{proof}

The angular Kronecker constant when $n\not\equiv a^{2}\mod (a+b)$ (i.e. $%
R\neq a$) now follows immediately.

\begin{corollary}
If $n\not\equiv a^{2}\mod (a+b)$, then $\alpha (a,b,n)=\beta (a,b,n)=E_{n}$.
\end{corollary}

\begin{proof}
We have already seen that $\alpha (a,b,n)\leq E_{n}$. Obviously, $\alpha
(a,b,n)\geq \beta (a,b,n)=E_{n}$, hence we have the equalities if $R\neq a.$
\end{proof}

\subsection{The angular Kronecker constants when $n\equiv a^{2}\mod (a+b)$.}

As with $n\not\equiv a^{2}\mod (a+b)$, we show that there is (up to $
\mod 1)$ one interval in which to search for the optimal approximation point.
However, in this case it will not suffice to consider only $\mathbf{t}\in
\{0,1/2\}^{3}$ as the angular Kronecker constant is greater than the binary
Kronecker constant.

Recall that for such $n$, $E_{n}=L_{n}$ where $L_{n}$ was defined in (\ref%
{Ln}).

\begin{lemma}
\label{singleInt} Assume $n\equiv a^{2}\mod (a+b)$ ($R=a)$ and that the real
numbers $x$, $t_{1}$, $t_{2}$ and integers $k_{1}$ and $k_{2}$ satisfy 
\begin{equation*}
\lambda _{x}:=ax-t_{1}-k_{1}=-(bx-t_{2}-k_{2})=\dfrac{1}{a+b}-L_{n}.
\end{equation*}%
Define $z_{1}=z_{1}(L_{n},x)$ and $z_{2}=z_{2}(L_{n},x)$ as in Lemma \ref%
{mainprelim}. For $z\in \lbrack z_{1},z_{2}]$, define $y_{z}$ as in (\ref%
{defy}). Then, for all real $t_{3}$, there is some $y\in \lbrack
y_{z_{1}},y_{z_{2}}]$ and integer $k_{3}$ such that 
\begin{equation*}
\mu (t_{1},t_{2},t_{3})=\Vert
(t_{1}+k_{1},t_{2}+k_{2},t_{3}+k_{3})-y(a,b,n)\Vert _{\infty }.
\end{equation*}
\end{lemma}

\begin{proof}
Since $\lambda _{x}=1/(a+b)-L_{n}$, we have $z_{2}-z_{1}=1/n$ (see Remark %
\ref{length} and the following discussion). Thus, given any $t_{3},$ there
is an integer $k_{3}$ such that $z=(t_{3}+k_{3})/n\in \lbrack z_{1},z_{2}]$.
Since $\lim_{n}L_{n}=1/(2a+2b),$ for large enough $n$ we have $\lambda
_{x}\geq (b-a)/(2n)$, thus the proof of Lemma \ref{mainprelim} shows that 
\begin{equation}
\mu (t_{1},t_{2},t_{3})\leq \Vert
(t_{1}+k_{1},t_{2}+k_{2},t_{3}+k_{3})-y_{z}(a,b,n)\Vert _{\infty }\leq L_{n}%
\text{.}  \label{sizeu}
\end{equation}

>From \cite{HR1}, we know there are $u\in \mathbb{R}$ and integers $j,k,m$
such that $\mu (t_{1},t_{2},t_{3})=\Vert
(t_{1},t_{2},t_{3})-u(a,b,n)+(j,k,m)\Vert _{\infty }$.

One can check that 
\begin{equation*}
\lambda _{x}=\frac{at_{2}-bt_{1}+ak_{2}-bk_{1}}{a+b}\text{.}
\end{equation*}%
Consequently, if $(t_{1}+j)/a\leq (t_{2}+k)/b,$\ then there is a
non-negative integer $H$ such that%
\begin{equation*}
\dfrac{t_{2}+k}{b}-\dfrac{t_{1}+j}{a}=\dfrac{H+(a+b)\lambda _{x}}{ab}.
\end{equation*}%
If $H\geq 1$, then arguments similar to those used as in Lemma \ref%
{singleInt2} show that $\mu (t_{1},t_{2},t_{3})\geq 1/(a+b)$. Hence $H=0$.
But then one can argue, as in Lemma \ref{singleInt2}, that there is an
integer $s$ such that if $u-s\notin \lbrack y_{z_{1}},y_{z_{2}}],$ then
either $\left\langle au-t_{1}\right\rangle >L_{n}$ or $\left\langle
bu-t_{2}\right\rangle >L_{n}$. Since we know from (\ref{sizeu}) that $\mu
(t_{1},t_{2},t_{3})\leq L_{n}$, this is a contradiction.

If, instead, $(t_{2}+k)/b<(t_{1}+j)/a$, then there is an integer $H\geq 1$
such that 
\begin{equation*}
\dfrac{t_{1}+j}{a}-\dfrac{t_{2}+k}{b}=\dfrac{H-(a+b)\lambda _{x}}{ab}.
\end{equation*}%
Let $y=(t_{1}+t_{2}+j+k)/(a+b)$. As 
\begin{equation*}
by-t_{2}-k=-(ay-t_{1}-j)=\dfrac{bt_{1}-at_{2}+bj-ak}{a+b}=\dfrac{H}{a+b}%
-\lambda _{x},
\end{equation*}%
it follows that if $u\geq y$, then $\langle bu-t_{2}\rangle \geq
H/(a+b)-\lambda _{x}$, while if $u\leq y$ then $\langle au-t_{1}\rangle \geq
H/(a+b)-\lambda _{x}$. Hence 
\begin{equation*}
\mu (t_{1},t_{2},t_{3})\geq \dfrac{1}{a+b}-\lambda _{x}=L_{n}
\end{equation*}%
But we already know that $\mu (t_{1},t_{2},t_{3})\leq L_{n}$ and this
approximation accuracy can be achieved by some $y\in \lbrack
y_{z_{1}},y_{z_{2}}].$
\end{proof}

\begin{proposition}
\label{keyLB} $\alpha (a,b,n)\geq L_{n}$ for all $n$ sufficiently large.
\end{proposition}

\begin{proof}
Because $a$ and $b$ are positive integers, the function $(t_{1},t_{2})%
\mapsto (at_{2}-bt_{1})/(a+b)$ is onto $\mathbb{R}$ from $\mathbb{R}^{2}$.
In particular, there are real numbers such that 
\begin{equation*}
(at_{2}-bt_{1})/(a+b)=1/(a+b)-L_{n}.
\end{equation*}%
Set $k_{1}=k_{2}=0$ and $x=(t_{1}+t_{2}+k_{1}+k_{2})/(a+b)$. Then 
\begin{equation}
\lambda _{x}:=ax-t_{1}-k_{1}=-(bx-t_{2}-k_{2})=\dfrac{%
at_{2}-bt_{1}+ak_{2}-bk_{1}}{a+b}=\dfrac{1}{a+b}-L_{n},  \label{LB}
\end{equation}%
so $x$, $t_{1}$, $t_{2}$, $k_{1}$ and $k_{2}$ satisfy the hypotheses of the
Lemma \ref{singleInt}. Let $z_{1}$ and $z_{2}$ be provided by that lemma.
Let $t_{3}=nz_{2}$. By Lemma \ref{singleInt} there is a real $u\in \lbrack
y_{z_{1}},y_{z_{2}}]$ and integer $k_{3}$ such that 
\begin{equation*}
\mu (t_{1},t_{2},t_{3})=\Vert
(t_{1}+k_{1},t_{2}+k_{2},t_{3}+k_{3})-u(a,b,n)\Vert _{\infty }.
\end{equation*}%
Because $\mu (t_{1},t_{2},t_{3})\leq \alpha (a,b,n)<1/2$, we know that $k_{1}
$ is the unique nearest integer to $t_{1}-ua$, the same for $k_{2}$ and $%
t_{2}-bu$, and for $k_{3}$ and $t_{3}-nu$. By the proof of Lemma \ref%
{mainprelim}, there is an integer $k$ such that the real number $y_{z_{2}}$
satisfies 
\begin{equation*}
L_{n}=ay_{z_{2}}-t_{1}-k_{1}=-(ny_{z_{2}}-t_{3}-k)\quad \text{and}\quad
z_{2}=(t_{3}+k)/n.
\end{equation*}%
By our choice of $t_{3}$ as $nz_{2}$, we know that $k=0$.

Suppose $z_{2}-1/(2n)\leq u<y_{z_{2}}$. Then $1/2\geq t_{3}-nu\geq 0$ and
thus $\langle t_{3}-nu\rangle =t_{3}-nu$. Consequently, 
\begin{equation*}
\langle t_{3}-nu\rangle =t_{3}-nu>t_{3}-ny_{z_{2}}=L_{n}\geq \mu
(t_{1},t_{2},t_{3}).
\end{equation*}%
This contradiction excludes $u$ from $[z_{2}-1/(2n),y_{z_{2}})$.

Suppose $y_{z_{1}}<u\leq z_{2}-1/(2n)$. As noted in the proof of Lemma \ref%
{mainprelim}, we have $y_{z_{1}}\geq z_{1}$. Furthermore, we have $1/2\leq
t_{3}-nu<1$ and thus $\langle t_{3}-nu\rangle =1-t_{3}+nu$. Using the fact
that $z_{1}=z_{2}-1/n$ and some details from the proof of Lemma \ref%
{mainprelim} gives the inequalities 
\begin{equation*}
\begin{aligned} \langle t_3-nu \rangle &=1-t_3+nu>1-nz_2+ny_{z_1} \\
&=1-(nz_1+1)+ny_{z_1}\ge L_n\ge\mu(t_1,t_2,t_3). \end{aligned}
\end{equation*}%
This contradiction excludes $u$ from $(y_{z_{1}},z_{2}-1/(2n)]$.

Thus $u=y_{z_{2}}$ or $u=y_{z_{1}}$. If $u=y_{z_{2}},$ then $k_{3}=k=0$ and $%
\mu (t_{1},t_{2},t_{3})=L_{n}$. Suppose $u=y_{z_{1}}$. Here $%
t_{3}=nz_{2}=nz_{1}+1$ and thus $z_{1}=(t_{3}-1)/n$. By the proof of Lemma %
\ref{mainprelim}, 
\begin{equation*}
L_{n}=t_{2}+k_{2}-by_{z_{1}}=ny_{z_{1}}-t_{3}-(-1)=ny_{z_{1}}-t_{3}+1\geq
|t_{1}+k_{1}-ay_{z_{1}}|.
\end{equation*}%
For $n$ large enough, $L_{n}<1/(a+b)<1/2$, thus $k_{3}=-1$ and 
\begin{equation*}
L_{n}=\Vert (t_{1}+k_{1},t_{2}+k_{2},t_{3}+k_{3})-y_{z_{1}}(a,b,n)\Vert
_{\infty }=\mu (t_{1},t_{2},t_{3}).
\end{equation*}
\end{proof}

\begin{corollary}
If $n\equiv a^{2}\mod (a+b)$, then $\alpha (a,b,n)=L_{n}$.
\end{corollary}

This completes the proof of Theorem \ref{T:three}.

\begin{remark}
>From the proof of Prop. \ref{keyLB} one can find rational $t_{1},t_{2},t_{3}$%
, depending only on $a,b,n$, so that the bound $\mu (t_{1},t_{2},t_{3})$ is
optimal. Indeed, the following choice will work: Let $t_{1}=0$, 
\begin{equation*}
t_{2}=\frac{a+b}{a}\left( \frac{1}{a+b}-L_{n}\right) \text{ and }t_{3}=n%
\frac{(a+n)(n+ab)}{2an(an+bn+ab)}\text{.}
\end{equation*}%
We note that when $x=t_{2}/(a+b)$, then $\lambda _{x}$ satisfies the
identities in (\ref{LB}) and simplifying gives%
\begin{equation*}
z_{2}:=x+\frac{(a+n)L_{n}-n\lambda _{x}}{an}=\frac{(a+n)(n+ab)}{2an(an+bn+ab)%
}.
\end{equation*}
Hence $t_{3}=nz_{2}$, so this choice of $t_{1},t_{2},t_{3}$ satisfies all
the requirements of the proof of the proposition.
\end{remark}

\end{document}